\documentclass[12pt]{article}
\usepackage{fourier}
\usepackage{index}
\usepackage[expansion=false]{microtype}
\usepackage[reqno]{amsmath}
\usepackage{fixmath}
\usepackage{amssymb, amsthm, enumerate}
\usepackage{mathtools}
\usepackage{pgf,tikz,pgfplots}
\usetikzlibrary{automata, positioning, arrows}
\usepackage[abbrev,msc-links]{amsrefs} 

\usepackage{hyperref}

\newtheoremstyle{plainsl}%
	{\topsep}
	{\topsep}
	{\slshape} 
	{}
	{\normalfont\bfseries}
	{.}
	{ }
	{}

\swapnumbers

{\theoremstyle{plainsl}
\newtheorem{theorem}{Theorem}[section]
\newtheorem{lemma}[theorem]{Lemma}
\newtheorem{corollary}[theorem]{Corollary}}
{\theoremstyle{remark}
\newtheorem{example}[theorem]{Example}}
\newtheorem{remark}[theorem]{Remark}
\newtheorem{question}[theorem]{Question}

\newcommand\cref[1]{Corollary~\ref{cor:#1}}

\newcommand\sqr[2]{{\vbox{\hrule height.#2pt
    \hbox{\vrule width.#2pt height#1pt \kern#1pt
        \vrule width.#2pt}\hrule height.#2pt}}}
\renewcommand\qed{%
	\ifmmode\eqno\sqr53
	\else\nolinebreak\ \hfill\sqr53\medbreak\fi}

\numberwithin{equation}{section}

\newcommand\al{\alpha}
\newcommand\be{\beta}

\newcommand\De{\Delta}

\newcommand\ga{\gamma}
\newcommand\Ga{\Gamma}

\newcommand\la{\lambda}

\newcommand\om{\omega}

\newcommand\sg{\sigma}

\renewcommand\th{\theta} 

\newcommand\cA{{\mathcal A}}

\newcommand\ints{{\mathbb Z}}
\newcommand\re{{\mathbb R}}
\newcommand\rats{{\mathbb Q}}

\newcommand\diff{\mathbin{\mkern-1.5mu\setminus\mkern-1.5mu}}

\newcommand\pmat[1]{\begin{pmatrix} #1 \end{pmatrix}}

\DeclareMathOperator{\diag}{diag}

\begin{document}

\title{Fractional revival on non-cospectral vertices}

\author{Chris Godsil\footnote{Research supported by 
Natural Sciences and Engineering Research Council of Canada, 
Grant No. RGPIN-9439},  Xiaohong Zhang\\
University of Waterloo}

\maketitle

\begin{abstract}
Perfect state transfer and fractional revival  can be used to move information
between pairs of vertices in a quantum network. While perfect state transfer has received a lot of  attention, fractional revival is newer and less studied. One problem is to determine the differences between perfect
state transfer and fractional revival. If perfect state transfer occurs between two vertices in a
graph, the vertices must be cospectral. Further if there is perfect state transfer between vertices
$a$ and $b$ in a graph, there cannot be perfect state transfer from $a$ to any other vertex.
No examples of unweighted graphs with fractional revival between non-cospectral vertices were known; here we give an infinite family of such graphs. No examples of unweighted graphs where the pairs involved in fractional revival overlapped were known; we give examples of such graphs as well.
\end{abstract}

\section{Introduction}

One important problem in quantum computing is to provide ways to transfer information from
one part of network of qubits to another. The issue is complicated by the fact that it is not possible
to copy a quantum state.

One approach to this is based on continuous quantum walks, as we now explain. Here we are
given a network of $n$ qubits, specified by a graph $X$ with adjacency matrix $A$.
Given $A$ we define the \textsl{transition matrix} $U(t)$ at time $t$ by 
\[
U(t) = \exp(itA).
\]
These matrices are unitary and act on a quantum system whose states are represented by
positive semidefinite matrices of size $n\times n$ and trace 1 (so-called \textsl{density matrices}).
If the initial state is given by $D$, the state of the system at time $t$ is $U(t)DU(-t)$.
This system is known as \textsl{continuous quantum walk}. (For more information, see \cites{thebook,drg,Periodic,godsil2017real,transfer}.)
We use $e_a$ to denote the standard basis vector indexed by the vertex $a$ in $X$
and set $D_a$ equal to $e_ae_a^T$. We have \textsl{perfect state transfer} from
vertices $a$ to vertex $b$ at time $t$ if $D_b=U(t)D_aU(-t)$; if this happens then
the state of the $b$-qubit at time $t$ is equal to the state of the $a$-qubit at time zero.  

Perfect state transfer takes place between vertices at distance $d$ in the $d$-cube,
as shown by Christandl et al. \cite{PST05}. Subsequent work yielded further examples \cites{circ,cubelike,drg},
but it also revealed that perfect state transfer is rare:  there are only
finitely many connected graphs with maximum valency $k$ that admit perfect state transfer \cite{transfer}.
Because of this, various extensions of the concept have been studied. Our concern in this paper is with fractional revival. 
We say that \textsl{fractional revival} occurs on the pair of vertices $\{a,b\}$ 
at time $t$, if 
\begin{equation}\label{eq:periab}
	U(t)(D_a+D_b)U(-t)=D_a+D_b.
\end{equation}
That is equivalent to say that $U(t)_{a,c}=U(t)_{b,c}=0$ for any $c\ne a,b$ (or $U(t)$ is  permutation-similar to a $2\times 2$ block diagonal matrix,  
with one diagonal block indexed by $a,b$). 
If fractional revival occurs on $\{a,b\}$ and  $U(t)D_aU(-t)\ne D_a$ (equivalently, $U(t)_{a,b}\ne 0$),  
then we say that \textsl{proper fractional revival} occurs on $\{a,b\}$.
It includes $U(t)D_aU(-t)= D_b$, 
that is, when there is perfect state transfer from $a$ to $b$,
as a  special case.

A number of papers \cites{frchain,frchain2,fr,gwithfr,FFR} have investigated fractional revival, providing examples
and developing the theory.
We discuss the contributions of this paper. There are two important properties of 
perfect state transfer:
\begin{enumerate}[(a)]
	\item 
	If there is perfect state transfer on $X$ from $a$ to $b$, then the vertex-deleted subgraphs
	$X\diff a$ and $X\diff b$ are cospectral (we say that $a$ and $b$ are cospectral
	vertices in $X$) \cite{transfer}.
	\item
	If there is perfect state transfer from $a$ to a second vertex, this vertex is unique \cite{pstmono}.
	(This is known as monogamy of perfect state transfer.)
\end{enumerate} 
Prior to our work, in all examples of proper fractional revival the two vertices involved were
cospectral. We provide an infinite family of examples where this is not the case.
Similarly, monogamy held in all known examples. We provide examples of graphs with
vertices $a$, $b$, $c$ such that proper fractional revival occurs on $\{a,b\}$ and $\{a,c\}$.
(Chan et al.\cite{FFR} constructed examples of weighted graphs where monogamy 
fails.)

Some of the tools we use are new. In particular we introduce the concept of support graphs,
graphs on eigenvalues of an underlying graph, and we also make use of matrix algebras associated
to induced subgraphs (induced algebras). 

\section{Preliminaries}

In this section, we associate a graph to a density matrix with respect to an underlying graph $X$,  
introduce a matrix algebra, and  
review some necessary conditions and a characterization for proper fractional revival to occur.

\subsection{Support graph}

Let $X$ be a graph on $n$ vertices with adjacency matrix $A$ and consider quantum walks on $X$. 
Assume $A$ has exactly $m$ distinct eigenvalues $\th_1,\ldots,\th_m$, 
with $E_1,\ldots, E_m$ being the corresponding projection matrices to the eigenspaces. 
Then $A=\sum_{r=1}^m \theta_r E_r$ is the spectral decomposition of $A$,
and $U(t)=\sum_{r=1}^m e^{it\theta_r}E_r$. 
Let $a,b\in V(X)$. 
If $(E_r)_{a,a}=(E_r)_{b,b}$ for all $r$, 
then $a$ and $b$ are cospectral (this is equivalent to the subgraphs $X\diff a$ and $X\diff b$ being  cospectral). 
If $E_re_a$ and $E_re_b$ are parallel vectors for each $r$, 
then we say $a$ and $b$ are \textsl{parallel}.

Let $D$ be a density matrix of size $n$. 
The \textsl{eigenvalue support} of $D$ is the set $\Phi_D=\{(\theta_r,\theta_s): E_rDE_s\ne0\}$. 
As $E_sDE_r=(E_rDE_s)^*$, 
we know that $(\theta_s,\theta_r)\in \Phi_D$ if and only if $(\theta_r,\theta_s)\in \Phi_D$. 
We define the \textsl{support graph} of $D$ (with respect to $X$) to be the graph that 
has the distinct eigenvalues $\{\theta_1,\ldots,\theta_m\}$ of $A$ as its vertices,
with vertices $\theta_r$ and $\theta_s$  adjacent if $(\theta_r,\theta_s)\in \Phi_D$.  
If $E_rDE_s\ne0$, then  neither $E_rD= 0$ nor $E_sD=0$, 
and so $E_rDE_r\ne 0$ and
$E_sDE_s\ne 0$, i.e., $(\theta_r,\theta_r)\in\Phi_D$ and $(\theta_s,\theta_s)\in\Phi_D$. 
Therefore this graph will normally have loops
(any vertex on an edge is on a loop). 

Let $F$ be the splitting field of the characteristic polynomial of $A$ over $\rats$, 
that is, $F=\rats(\theta_1,\ldots,\theta_m)$. 
Denote the Galois group of $F$ over $\rats$ by $\Ga$. 
For any matrix $M\in \operatorname{Mat}_{n\times n}(F)$ and $\sg\in\Ga$,
let $M^{\sg}$ denote the matrix obtained by applying $\sg$ to each entry of $M$. 
Let $\theta_r$ be any eigenvalue of $A$,
then $\theta_r^\sg$ is an algebraic conjugate of $\theta_r$ 
and is an eigenvalue of $A$. 
In fact, 
\[
AE_r^\sg=(AE_r)^\sg=\theta_r^\sg E_r^\sg.
\]
Therefore, 
$\theta_r^\sg$ is also a vertex of the support graph. 
Now let  $D$  be a rational state, that is, all the entries of $D$ are rational numbers. 
For any two vertices $\theta_r,\theta_s$ of the support graph and any $\sg\in\Ga$, 
$E_r^{\sg}DE_s^{\sg}=(E_rDE_s)^{\sg}$. 
Therefore $E_r^{\sg}DE_s^{\sg}\ne 0$ if and only if $E_rDE_s\ne 0$. 
Hence in the support graph, 
$(\theta_r,\theta_s)$ is an edge if and only if $(\theta_r^\sg,\theta_s^\sg)$ is one; 
$\Ga$ acts as a group of automorphisms
of the support graph of a rational state $D$.
When $D=D_a=e_ae_a^T$, 
the eigenvalue support of $D$ is $\{(\theta_r,\theta_s): E_re_a\ne0 \text{ and }E_se_a\ne0\}$, 
which is usually equivalently defined as $\{\theta_r: E_re_a\ne0\}$. 

Let $D_1$ and $D_2$ be two density matrices of size $n$.  
If $U(t)D_1U(-t)=D_2$ for some $t>0$, 
then we say there is \textsl{perfect state transfer} from $D_1$ to $D_2$ at time $t$.
If $U(t)D_1U(-t)=D_1$ for some $t>0$, 
then we say $D_1$ is \textsl{periodic} at time $t$. 
Note that $D_1$ is periodic at time $t$ if and only if $U(t)$ and $D_1$ commute 
and that fractional revival on $\{a,b\}$ at time $t$ is equivalent to the state $\frac12 (D_a+D_b)$ being periodic at time $t$.  
Since unitarily similar matrices have the same trace, 
sometimes we represent a state with a positive semidefinite matrix instead of normalizing it to a density matrix (scale the matrix to one of trace 1), 
which does not change the analysis.

If $D$ is periodic at time $\tau$, then
\[
	D = U(\tau)DU(-\tau) = \sum_{r,s} e^{i\tau(\theta_r-\theta_s)}E_rDE_s
\]
and if we multiply on the left by $E_r$ and on the right by $E_s$, we get
\[
	E_rDE_s = e^{i\tau(\theta_r-\theta_s)}E_rDE_s.
\]
We conclude that if $(\theta_r,\theta_s)\in\Phi_D$,
then $e^{i\tau\theta_r}=e^{i\tau\theta_s}$.
Therefore if $D$ is periodic at time $\tau$, the function $e^{i\tau z}$ is constant on 
the vertices in a connected component of the support graph. 

For $S\subseteq V(X)$, 
define $D_S$ to be the diagonal matrix with $(a,a)$-entry equal to 1 if $a\in S$ and equal to 0 otherwise. 
Note that $D_S$ (again, we ignore the factor $\frac{1}{|S|}$) is a rational state. 
If $D_S$ is periodic (sometimes we say $S$ is periodic for simplicity), 
then its eigenvalue support has a special form:

\begin{theorem}\cite{godsil2017real}\label{thm:ratstran}
Let $X$ be a simple graph and 
$D$ be a rational state.  
If $D$ is periodic relative to the continuous quantum walk on $X$ at time $\tau$, 
then there is a square-free integer $\De$ such that 
$\frac{\theta_r-\theta_s}{\sqrt{\De}}\in\ints$ for all $(\theta_r,\theta_s)$ in the support of $D$.
\end{theorem}

Let $D$ be a periodic rational state. 
The above Theorem tells us that there exists a square-free integer $\De$ such that for any two vertices $\theta_r,\theta_s$ in the same connected component of the support graph of $D$, 
$\theta_r-\theta_s$ is an integer multiple of $\sqrt{\De}$. 
In particular, 
when $D=D_a+D_b$ is periodic, 
we have a necessary condition on $\Phi_D$ for fractional revival on $\{a,b\}$ to occur, 
as shown in Theorem~\ref{thm:FR}.

\subsection{Induced algebra}

Now assume 
without loss of generality that $S= \{1,2,\ldots,k\}$ is periodic at time $\tau$: 
$U(\tau)D_S=D_SU(\tau)$. 
Then $U(\tau)$ has form
\[
	\pmat{U_0&0\\ 0&U_1},
\]
where $U_0$ is $k \times k$ and 
\[
	D_SU(\tau)D_S = \pmat{U_0&0\\0&0}.
\]
Now
\begin{equation}
	\label{DUDeq:}
	(D_SU(\tau)D_S)(D_SE_rD_S) = D_SU(\tau)E_rD_S = e^{i\tau\theta_r}D_SE_rD_S
\end{equation}
and similarly $(D_SE_rD_S)(D_SU(\tau)D_S)=e^{i\tau\theta_r}D_SE_rD_S$;  
it follows that $D_SU(\tau)D_S$ and $D_SE_rD_S$ commute.

As before, 
let $X$ be a connected graph on $n$ vertices with adjacency matrix $A=\sum_{r=1}^m\theta_r E_r$. 
Let $M$ be a matrix indexed by $V(X)$. 
For a subset $S$ of $V(X)$, 
denote the submatrix of $M$ with rows and columns indexed by elements of $S$ by $M_{S,S}$. 
The \textsl{induced adjacency algebra} of $X$ on $S$ \cite{thebook} is the matrix algebra generated by the matrices 
\[
(E_r)_{S,S},\quad r=1,\ldots, m.
\]
Denote it by $\cA(S)$ and note that it is isomorphic to the algebra generated by the $n\times n$ matrices $D_SE_rD_S, \; r=1,\ldots, m$, 
and \eqref{DUDeq:} is equivalent to 
\begin{equation}\label{eq:indueige}
U_0 (E_r)_{S,S}=e^{i\tau\theta_r}(E_r)_{S,S}. 
\end{equation}
Therefore \eqref{eq:indueige} implies the following: 
\begin{lemma}\label{lem:pericenter}
If the subset $S$ of $V(X)$ is periodic relative to the continuous quantum walk on $X$ at time $\tau$, 
then $U(\tau)_{S,S}$ belongs to the center of $\cA(S)$, the induced algebra generated by $(E_r)_{S,S}$. \qed
\end{lemma}

If $M$ is a matrix with all eigenvalues being simple, 
then any matrix that commutes with $M$ is a polynomial in $M$. 

\begin{corollary}\label{coro:commalg}
Assume $S$ is a subset of $V(X)$ that is periodic relative to the continuous quantum walk on $X$ at time $\tau$. 
If the eigenvalues of  $U(\tau)_{S,S}$ are distinct, 
then the induced algebra $\cA(S)$ is commutative. \qed
\end{corollary}

Now consider the case when $D_S$ is periodic for $S=\{1,2\}$. 
If $U_0$ is a diagonal but not a scalar matrix,
then the eigenvalues of $U_0$ are distinct and 
therefore $(E_r)_{S,S}$ is diagonal.
It follows that $(A^k)_{1,2}=\sum_r \theta_r^k(E_r)_{1,2}=0$ for all positive integer $k$, 
and therefore $X$ is not connected, a contradiction. 
Hence if  non-proper fractional revival occurs on $\{1,2\}$ at time $\tau$, 
then $U_0$ is a scalar matrix and so we have simultaneous 
periodicity at vertices 1 and 2 with the same phase factor. 
When proper fractional revival occurs, 
the eigenvalues of $U_0$ are distinct (the only $2\times 2$ diagonalizable matrix with repeated eigenvalues is a scalar matrix), 
and we obtain another necessary condition for fractional revival.

\begin{lemma}\cite{FFR}
		If proper fractional revival occurs on $S=\{a,b\}$, 
		then the induced algebra $\cA(S)$ is commutative.\qed
\end{lemma}

\smallbreak\noindent

As mentioned, prior to this work, 
in all examples of fractional revival the two vertices involved, say $a$ and $b$,
were cospectral, 
a condition stronger than requiring the induced adjacency algebra $\cA(\{a,b\})$ to be commutative. 
Similarly to cospectrality of vertices $a,b$ \cite{strongcos}, 
 commutativity of the induced algebra $\cA(\{a,b\})$   has some equivalent combinatorial descriptions. 

\begin{theorem}\cite{FFR} \label{thm:commuindal}
Let $X$ be a connected graph and let $A=\sum_{r}\theta_r E_r$ be the spectral 
decomposition of $A$. 
For vertices $a,b$ in $X$, 
the following  are equivalent:
\begin{enumerate}[(a)]
	\item 
	The induced algebra $\cA(\{a,b\})$ is commutative, 
	\item
	there exists some $\ga$ such that 
	$(E_r)_{a,a}-(E_r)_{b,b} = \ga (E_r)_{a,b}$ for all $r$, 
	\item
	There exists some $\ga$ such that 
	$(A^k)_{a,a}-(A^k)_{b,b} =  \ga  (A^k)_{a,b}$ for all positive integer $k$, 
	\item
	For some $\ga$
	\[
		\phi(X\diff a,t)-\phi(X\diff b,t)
			=\ga\sqrt{\phi(X\diff a,t)\phi(X\diff b,t)-\phi(X,t)\phi(X\diff\{a,b\},t)}, 
	\]
\end{enumerate}
\end{theorem}
\noindent
where for $S\subseteq V(X)$, $X\diff S$ denotes the induced subgraph of $X$ on $S$ and
$\phi(X\diff S)$ denotes the characteristic polynomial of the adjacency matrix $A(X\diff S)$. 
The value of $\ga$ is the same in each case, and is rational. We see that
$\ga=0$ if and only if $a$ and $b$ are cospectral. 
If any of the above condition holds, 
vertices $a$ and $b$ are said to be \textsl{fractionally cospectral} \cite{FFR}. 

We use $\psi_{a,b}(X,t)$ to denote 
$
	\sqrt{\phi(X\diff a,t)\phi(X\diff b,t)-\phi(X,t)\phi(X\diff\{a,b\},t)}
$
(even if $\cA(\{a,b\})$ is not commutative).

Theorem~\ref{thm:commuindal} implies that if $a$ and $b$ satisfy certain relations or if $X$ is of certain type, 
then the commutativity of the induced algebra $\cA(\{a,b\})$  implies that $a$ and $b$ are cospectral. 
\begin{corollary}\cite{FFR}
Let $X$ be a simple graph and $a,b\in V(X)$. 
Assume the induced adjacency algebra $\cA(\{a,b\})$ is commutative.
If one of the following conditions holds, then $a$ and $b$ are cospectral: 
\begin{itemize}
\item
$a$ is adjacent to $b$
\item
$X$ is bipartite, and the distance between $a$ and $b$ is odd
\item
$a$ and $b$ are of the same degree and are at distance two in $X$
\item
$X$ is a connected regular graph
\end{itemize}
\end{corollary}

Now we derive one more necessary condition for proper fractional revival to occur. 
Recall that when proper fractional revival occurs on $S=\{a,b\}$ at time $\tau$, 
the two eigenvalues of $U_0$ are both simple. 
Equation \eqref{eq:indueige} implies that for each $r$, the non-zero columns
of $(E_r)_{S,S}$ are eigenvectors of $U(\tau)_{S,S}$ associated to eigenvalue $e^{i\tau\theta_r}$, 
therefore $(E_r)_{S,S}$ has rank at most 1 
and vertices $a$ and $b$ are parallel. 
In fact, 
from $\det\left((E_r)_{S,S}\right)=(E_r)_{a,a}(E_r)_{b,b}-(E_r)_{a,b}(E_r)_{b,a}$ and 
$(E_r)_{a,b}=\langle E_re_a, E_re_b \rangle$ 
(where $\langle,\rangle$ denotes the inner product of vectors), 
we know that $\det(E_r)_{S,S}=0$ if and only if 
\[
\langle E_re_a, E_re_a \rangle \langle E_re_b, E_re_b \rangle=|\langle E_re_a, E_re_b \rangle|^2, 
\]
which holds if and only if $E_re_a$ and $E_re_b$ are parallel. 
Since it holds for all $r$, vertices $a$ and $b$ are parallel. 

We have obtained three necessary conditions for proper fractional revival to occur: 
$\theta_r-\theta_s=m_{r,s}\sqrt{\De}$ for $(\theta_r,\theta_s)\in \Phi_{D_S}$ (Theorem~\ref{thm:ratstran}), 
commutativity of $\cA(S)$ (Corollary~\ref{coro:commalg}), and parallelity of vertices $a$ and $b$. 
To get a deeper description of the eigenvalue support of $D_S$, 
we explore the above information further. 
Assume that $\la_1$ and $\la_2$ are the two distinct eigenvalues of $U_0$, 
 $x$ is a unit eigenvector associated to eigenvalue $\la_1$, and that $y$ is a unit eigenvector associated to $\la_2$. 
Then for each $r$, either $(E_r)_{S,S}=\al_r xx^*$ if $e^{i \tau \theta_r}=\la_1$ or $(E_r)_{S,S}=\al_r yy^*$ if $e^{i \tau \theta_r}=\la_2$, for some nonnegative number $\al_r$,  
as $(E_r)_{S,S}$ is positive semidefinite of rank at most 1. 
Since  $(E_r)_{S,S}$ is a real matrix, 
the two eigenvectors $x$ and $y$ can be chosen to be real. 
Therefore $\langle x, y \rangle =0$, 
as real eigenvectors associated to distinct eigenvalues of a complex symmetric matrix are orthogonal.
Therefore for any two eigenvalues $\theta_r$ and $\theta_s$ of $A$, 
either $e^{i\tau \theta_r}=e^{i \tau \theta_s}$ and $(E_r)_{S,S}$ and $(E_s)_{S,S}$ are scalar multiple of each other, 
or $e^{i\tau \theta_r}\ne e^{i \tau \theta_s}$ and $(E_r)_{S,S}(E_s)_{S,S}=0$. 
Together with the fact $(\theta_r,\theta_s)\in \Phi_{D_S}$ if and only if $(E_r)_{S,S}(E_s)_{S,S}\neq 0$, this implies:  

\begin{lemma}\label{lem:profrsuppg}
If proper fractional revival occurs on $\{a,b\}$ in a graph $X$, 
then the support graph of $D_{\{a,b\}}$ has two components isomorphic to complete graphs
with loops, and the remaining components are isolated vertices (loopless).
\end{lemma}

In fact, 
the two nontrivial components of the support graph can be determined by the sign of $(E_r)_{a,b}$.  
Note that neither $x$ nor $y$ has zero entries. 
For otherwise $x^Ty=0$ implies that $\{x,y\}= \{[1 \; 0]^T, [0 \; 1]^T\}$, 
contradicting proper fractional revival occurs on $\{a,b\}$. 
Therefore either $x$ or $y$ has its two components both positive or both negative. 
Without loss of generality, 
assume $x = \begin{bmatrix}p \\q\end{bmatrix}$ with $pq > 0$ and $p^2+q^2=1$, 
 then  
$y =\pm \begin{bmatrix}-q \\p\end{bmatrix}.$ 
Therefore for some non-negative integer $\al_r$, 
\[
(E_r)_{S,S}=\al_{r}\begin{bmatrix}p^2 & pq\\ pq & q^2\end{bmatrix},\; \text{ or }
(E_r)_{S,S}=\al_{r}\begin{bmatrix}q^2 & -pq\\ -pq & p^2\end{bmatrix}, 
\] 
depending on whether $e^{i\tau \theta_r}=\la_1$
or $e^{i\tau \theta_r}=\la_2$. 
Furthermore, 
the scalar $\ga$ in Theorem~\ref{thm:commuindal} satisfies 
$\ga=\frac pq-\frac qp$.

Let $C^+$ be the set of indices $r$ such that $(E_r)_{a,b}>0$, 
and let $C^-$ be the set of indices $r$ such that $(E_r)_{a,b}<0$. 
Then $C^+$ and $C^-$ correspond to the two nontrivial components of the support graph of  $D_{\{a,b\}}$. Furthermore, 
\[
	U_0 = \sum_{r\in C^+} \al_r e^{it\theta_r} xx^T 
			+ \sum_{r\in C^-} \al_r e^{it\theta_r} yy^T
\]
implies that 	$\sum_{r\in C^+} \al_r = \sum_{r\in C^-} \al_r = 1$, 
as the eigenvalues of a unitary matrix are numbers on the unit circle. 
Therefore 
\begin{equation}\label{eq:profrspec}
 U_0=e^{i\tau \theta_r}\begin{bmatrix}p^2 & pq\\pq & q^2\end{bmatrix}
 +e^{i\tau \theta_s}\begin{bmatrix}q^2 & -pq\\-pq & p^2\end{bmatrix}.
\end{equation}
These three conditions ($\theta_r-\theta_s=m_{r,s}\sqrt{\De}$ for $(\theta_r,\theta_s)\in \Phi_{D_S}$, 
commutativity of $\cA(S)$, and parallelity of vertices $a$ and $b$) 
  plus one more condition (that ensures the fractional revival is proper, that is, 
  $e^{i\tau \theta_r}\ne e^{i\tau \theta_s}$) in fact form a characterization of when proper fractional revival occurs on a graph.

\begin{theorem}\cite{FFR} \label{thm:FR}
Let $X$ be a connected graph with adjacency matrix $A=\sum_r \theta_r E_r$. Proper fractional revival occurs on $\{a,b\}$ in $X$ if and only if the following conditions hold.
\begin{enumerate}[(a)]
\item
$a$ and $b$ are parallel
\item
The induced algebra on $\{a,b\}$ is commutative
\item 
Let $C^+$ be the set of indices $r$ such that $(E_r)_{a,b}>0$, 
and let $C^-$ be the set of indices $r$ such that $(E_r)_{a,b}<0$. 
Then there exist a square-free integer $\De$ such that 
\begin{align*}
&\theta_j=\rho_j\sqrt{\De}, \text{ for } j\in C^+, \text{ and}\\
&\theta_{\ell}=\om_{\ell}\sqrt{\De}, \text{ for } \ell\in C^-,
\end{align*}
where $\rho_j$'s and $\om_{\ell}$'s are real numbers satisfying 
\[
 \rho_j-\rho_{j'} \in \mathbb{Z} \text{ and } \om_{\ell}-\om_{\ell'}\in\mathbb{Z}\]
for all $j,j'\in C^+$ and $\ell, \ell'\in C^-$.
\item
Let $g=\gcd\left\{\frac{\theta_r-\theta_s}{\sqrt{\De}}\; : \, r,s\in C^+, \text{ or } r,s\in C^-\right\}$. 
There exists $j\in C^+$ and $\ell \in C^-$ such that 
\[
\frac{\rho_j-\om_{\ell}}{g}\notin \mathbb{Z}.
\]
\end{enumerate}
Moreover, if these conditions hold, then $\frac{2\pi}{g\sqrt{\De}}$ is the minimum time for proper fractional revival to occur on $\{a,b\}$ in $X$. 
Note that  fractional revival on $\{a,b\}$ occurs if and only if conditions $(a),(b)$ and $(c)$ hold.
\qed
\end{theorem}

\subsection{Equitable partition}
We review about equitable partitions of a graph \cite{algcomb}. 
Let $X$ be a graph on $n$ vertices. 
Let $\mathcal{P}=\{C_1,\ldots, C_k\}$ be a partition of $V(X)=\{1,\ldots, n\}$. 
The \textsl{characteristic matrix} $P$ of $\mathcal{P}$ is an $n\times k$ matrix
with $p_{j\ell}=1$ if $j\in C_{\ell}$ and $p_{j\ell}=0$ otherwise. 
The \textsl{normalized characteristic matrix} of the partition $\mathcal{P}$ is 
\[
\hat{P}=P\diag(\frac{1}{\sqrt{|C_1|}},\cdots, \frac{1}{\sqrt{|C_k|}}). 
\]
Note that $\hat{P}^T \hat{P}=I_k$ and $\hat{P}\hat{P}^T$ is the orthogonal projection onto the subspace of $\re^n$ consisting of vectors that are constant on cells of $\mathcal{P}$.
A partition $\mathcal{P}=\{C_1,\ldots, C_k\}$ is \textsl{equitable} if 
for any $j,\ell$, the number of neighbors which a vertex in $C_j$ has 
in $C_{\ell}$ is independent of the choice of vertex in $C_j$. 
Denote this fixed number of neighbors  as $c_{j\ell}$. 
Given an equitable partition $\mathcal{P}=\{C_1,\ldots, C_k\}$ of a graph $X$, 
a weighted graph can be associated to it. 
Let $\widehat{X/ \mathcal{P}}$ be the weighted graph with 
the cells of $\mathcal{P}$ as its vertices, 
with the  edge between $C_j$ and $C_{\ell}$ of weight $\sqrt{c_{j\ell}c_{\ell j}}$ (if $C_{j\ell}=0$ then  $C_j$ and $C_{\ell}$ are not adjacent). 
$\widehat{X/ \mathcal{P}}$ is called the \textsl{symmetrized quotient graph} of $X$ with respect to $\mathcal{P}$. 

Eigenvectors of $A(X)$ and of $A(\widehat{X/ \mathcal{P}})$ are related. 
\begin{lemma}\cite{algcomb}\label{lem:equeign}
Let $\mathcal{P}$ be an equitable partition of the graph $X$ with $k$ cells. 
Assume $\hat{P}$ is the normalized characteristic matrix of $\mathcal{P}$. 
Then
\begin{enumerate}[(a)]
\item
If $A(\widehat{X/ \mathcal{P}})x=\theta x$, 
then $A(X)\hat{P}x=\theta \hat{P}x$.
\item
If $A(X)y=\theta y$, 
then $A(\widehat{X/ \mathcal{P}})\hat{P}^Ty=\theta \hat{P}^Ty$
\end{enumerate}
\end{lemma}
\noindent
Given an orthonormal basis
 $x_1,\ldots, x_k$ of $\re^k$ that consists of eigenvectors of 
 $A(\widehat{X/\mathcal{P}})$,  part $(a)$ of the above lemma tells us that 
$
\hat{P}x_1, \ldots, \hat{P}x_k
$
 is a list of orthonormal eigenvectors of $A(X)$: 
 \[\langle \hat{P}x_j,\hat{P}x_{\ell}\rangle=\langle x_j,\hat{P}^T\hat{P}x_{\ell}\rangle=\langle x_j,I_kx_{\ell}\rangle=\delta_{j,\ell}.
 \]
We can further extend the list to an orthonormal basis of $\re^n$  consisting of eigenvectors of $A(X)$, say 
$\hat{P}x_1,\ldots, \hat{P}x_k, y_{k+1},\ldots, y_n$. 
Then for any $j\in\{k+1,\ldots, n\}$, 
we have $\hat{P}^Ty_j=0$. 
In fact $\langle \hat{P}^Ty_j, x_{\ell}\rangle =\langle y_j,\hat{P} x_{\ell}\rangle =0$
for all $\ell\in \{1,\ldots, k\}$ implies that $\hat{P}^Ty_j$ lies in the orthogonal complement of $\operatorname{Span}\{x_1,\ldots, x_k\}$, which is $\{0\}$.

\section{Stellar Fusion}\label{sec:FR}
Let $a, k$ and $c$ be positive integers. 
Let $X=X(a,k,c)$ denote the graph obtained by taking a copy of the star $K_{1,a+k}$ and a copy of the star $K_{1,c+k}$ and merging $k$ vertices of degree one from the first star with $k$ vertices of degree one from the second. 
In this section we only consider graphs of such type. 

Let $\mathcal{P}$ be a partition of $V(X)$, where the two centers of the two stars, denoted by 0 and 1, respectively, are both singletons of $\mathcal{P}$, the $k$ common neighbours of 0 and 1 form a cell, the remaining $a$ neighbours of $0$ form a cell, and the $c$ remaining neighbours of vertex 1 form a cell, as shown in Figure~\ref{fig:graphX}.   
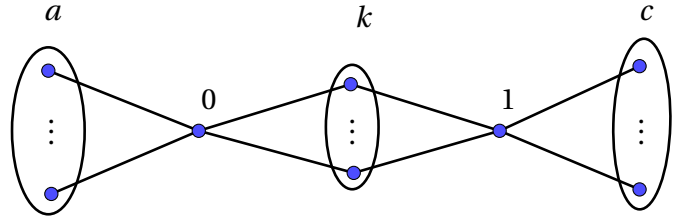
\begin{figure}[h]
\begin{center}
\definecolor{ududff}{rgb}{0.30196078431372547,0.30196078431372547,1.}
\begin{tikzpicture}[line cap=round,line join=round,>=triangle 45,x=1.0cm,y=1.0cm,scale=1.0]
\clip(0,0.8) rectangle (10.04,3.7);
\draw [rotate around={90.:(1.,2.)},line width=1.pt] (1.,2.) ellipse (1.1104616068093949cm and 0.48283017738921946cm);
\draw [rotate around={90.:(5.04,2.05)},line width=1.pt] (5.04,2.05) ellipse (0.8277224665102346cm and 0.3501777856544038cm);
\draw [rotate around={88.39399701071756:(8.89,2.09)},line width=1.pt] (8.89,2.09) ellipse (1.1353270165076959cm and 0.37837472750187456cm);
\draw [line width=1.pt] (1.,2.8)-- (3.,2.);
\draw [line width=1.pt] (1.04,1.16)-- (3.,2.);
\draw [line width=1.pt] (3.,2.)-- (5.02,2.62);
\draw [line width=1.pt] (3.,2.)-- (5.06,1.44);
\draw [line width=1.pt] (5.06,1.44)-- (7.,2.);
\draw [line width=1.pt] (5.02,2.62)-- (7.,2.);
\draw [line width=1.pt] (7.,2.)-- (8.86,2.86);
\draw [line width=1.pt] (7.,2.)-- (8.86,1.22);

\draw (0.8,3.8) node[anchor=north west] {$a$};
\draw (4.94,3.8) node[anchor=north west] {$k$};
\draw (8.72,3.8) node[anchor=north west] {$c$};
\draw (0.86,2.5) node[anchor=north west] {$\vdots$};
\draw (4.86,2.5) node[anchor=north west] {$\vdots$};
\draw (8.7,2.5) node[anchor=north west] {$\vdots$};
\draw (2.9,2.68) node[anchor=north west] {0};
\draw (6.88,2.68) node[anchor=north west] {1};

\begin{scriptsize}
\draw [fill=ududff] (3.,2.) circle (2.5pt);
\draw [fill=ududff] (7.,2.) circle (2.5pt);
\draw [fill=ududff] (1.,2.8) circle (2.5pt);
\draw [fill=ududff] (1.04,1.16) circle (2.5pt);
\draw [fill=ududff] (5.02,2.62) circle (2.5pt);
\draw [fill=ududff] (5.06,1.44) circle (2.5pt);
\draw [fill=ududff] (8.86,2.86) circle (2.5pt);
\draw [fill=ududff] (8.86,1.22) circle (2.5pt);
\end{scriptsize}
\end{tikzpicture}
\caption{Graph $X=X(a,k,c)$ obtained from two stars with parameters $a$, $k$ and $c$, and an equitable partition of it}
\label{fig:graphX}
\end{center}
\end{figure}
This is an equitable partition, and the corresponding symmetrized quotient graph is a weighted path on 5 vertices, as shown in Figure~\ref{fig:quotientX}, which we denote by $\hat{X}$.

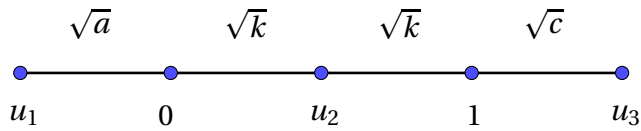
\begin{figure}[h]
\begin{center}
\definecolor{ududff}{rgb}{0.30196078431372547,0.30196078431372547,1.}
\begin{tikzpicture}[line cap=round,line join=round,>=triangle 45,x=1.0cm,y=1.0cm]
\clip(0,1.3) rectangle (10.04,2.9);
\draw [line width=1.pt] (1.,2.)-- (3.,2.);
\draw [line width=1.pt] (3.,2.)-- (5.,2.);
\draw [line width=1.pt] (5.,2.)-- (7.,2.);
\draw [line width=1.pt] (7.,2.)-- (9.,2.);
\draw (1.5,3.) node[anchor=north west] {$\sqrt{a}$};
\draw (3.58,3.) node[anchor=north west] {$\sqrt{k}$};
\draw (5.62,3.) node[anchor=north west] {$\sqrt{k}$};
\draw (7.56,3.) node[anchor=north west] {$\sqrt{c}$};
\draw (2.7,1.7) node[anchor=north west] {$0$};
\draw (6.8,1.7) node[anchor=north west] {$1$};
\draw (0.7,1.7) node[anchor=north west] {$u_1$};
\draw (4.7,1.7) node[anchor=north west] {$u_2$};
\draw (8.7,1.7) node[anchor=north west] {$u_3$};
\begin{scriptsize}
\draw [fill=ududff] (1.,2.) circle (2.5pt);
\draw [fill=ududff] (3.,2.) circle (2.5pt);
\draw [fill=ududff] (5.,2.) circle (2.5pt);
\draw [fill=ududff] (7.,2.) circle (2.5pt);
\draw [fill=ududff] (9.,2.) circle (2.5pt);
\end{scriptsize}
\end{tikzpicture}
\caption{the symmetrized quotient graph $\hat{X}$ of $X$ with respect to partition $\mathcal{P}$}
\label{fig:quotientX}
\end{center}
\end{figure}

\begin{lemma}\label{lem:SFC}
For any positive integers $a,k,c$, 
vertices 0 and 1 are parallel and 
the induced algebra of $X=X(a,k,c)$ on $\{0,1\}$  is commutative. 
Moreover, if $C^+$ and $C^-$ are defined as in Theorem~\ref{thm:FR} 
and $\mu = 2k+a+c$ and $\sg=4k^2+(a-c)^2$, 
then $C^+=\left\{ -\sqrt{\frac{\mu + \sqrt{\sg}}{2}},
 \sqrt{\frac{\mu + \sqrt{\sg}}{2}}\right\}$ and 
$C^-=\left\{-\sqrt{\frac{\mu - \sqrt{\sg}}{2}},
\sqrt{\frac{\mu - \sqrt{\sg}}{2}}\right\}$. \end{lemma}
\begin{proof}
Let $\hat{X}$ be the symmetrized quotient graph of $X$ with respect to the above equitable partition $\mathcal{P}$. Let $\mu = 2k+a+c$ and $\sg=4k^2+(a-c)^2$. 
Then the eigenvalues of $\hat{X}$ are 
\begin{align}\label{eq:eigenvalues}
&\theta_1=0, \nonumber\\ 
&\theta_2= -\sqrt{\frac{\mu - \sqrt{\sg}}{2}}, \quad
\theta_3= \sqrt{\frac{\mu - \sqrt{\sg}}{2}},\\
&\theta_4= -\sqrt{\frac{\mu + \sqrt{\sg}}{2}},\quad
\theta_5= \sqrt{\frac{\mu + \sqrt{\sg}}{2}}.\nonumber
\end{align}
with corresponding eigenvectors (with vertices ordered as $u_1,0,u_2,1,u_3$)
\[
v_1=\left[
\frac{\sqrt{c}}{\sqrt{a}}, 0, -\frac{\sqrt{c}}{\sqrt{k}}, 0, 1\right]^T
\]
\begin{equation}\label{eq:eigvec}
v_2=\begin{bmatrix}-\frac{c \sqrt{a} - a^{3/2} + \sqrt{a} \sqrt{\sg}}{2 k \sqrt{c}}\\
 \frac{\sqrt{\mu - \sqrt{\sg}} \left(c-a + \sqrt{\sg}\right)}{2 \sqrt{2} k \sqrt{c}}\\
-\frac{-2 k + c - a + \sqrt{\sg}}{2 \sqrt{kc}}\\
-\frac{\sqrt{\mu - \sqrt{\sg }}}{\sqrt{2c}}\\
1\end{bmatrix},
v_3=\begin{bmatrix}-\frac{c \sqrt{a} - a^{3/2} + \sqrt{a} \sqrt{\sg}}{2 k \sqrt{c}}\\
- \frac{\sqrt{\mu - \sqrt{\sg}} \left(c-a + \sqrt{\sg}\right)}{2 \sqrt{2} k \sqrt{c}}\\
-\frac{-2 k + c - a + \sqrt{\sg}}{2 \sqrt{kc}}\\
\frac{\sqrt{\mu - \sqrt{\sg }}}{\sqrt{2c}}\\
1\end{bmatrix},
\end{equation}
\[
v_4=\begin{bmatrix}-\frac{c \sqrt{a} - a^{3/2} - \sqrt{a} \sqrt{\sg}}{2 k \sqrt{c}}\\
 -\frac{\sqrt{\mu + \sqrt{\sg}} \left(-c +a + \sqrt{\sg}\right)}{2 \sqrt{2} k \sqrt{c}}\\
-\frac{-2 k + c - a - \sqrt{\sg}}{2 \sqrt{kc}}\\
-\frac{\sqrt{\mu + \sqrt{\sg}}}{\sqrt{2c}}\\ 
1\end{bmatrix},
v_5=\begin{bmatrix}-\frac{c \sqrt{a} - a^{3/2} - \sqrt{a} \sqrt{\sg}}{2 k \sqrt{c}}\\
 \frac{\sqrt{\mu + \sqrt{\sg}} \left(-c +a + \sqrt{\sg}\right)}{2 \sqrt{2} k \sqrt{c}}\\
-\frac{-2 k + c - a - \sqrt{\sg}}{2 \sqrt{kc}}\\
\frac{\sqrt{\mu + \sqrt{\sg}}}{\sqrt{2c}}\\ 
1\end{bmatrix}.
\]

As the characterization of proper fractional revival in Theorem~\ref{thm:FR} is for unweighted graphs (in fact it works for integer weighted graphs but not for all weighted graphs), 
we now use Lemma~\ref{lem:equeign} and the comments below it to further get eigenvector information of $A(X)$. 
First let $\hat{P}$ be the normalized characteristic matrix of the equitable partition $\mathcal{P}$ and for $j=1,\ldots 5$ let $(\theta_j,v_j)$ be an eigenpair of $A(\hat{X})$  as above with $v_j$s scaled to have norm 1. 
Then $u_j=\hat{P}v_j$ is an eigenvector of $A(X)$ associated to $\theta_j$, $j=1,\ldots, 5$. 
Now assume that $u_1,\ldots, u_5,u_6,\ldots, u_{a+k+c+2}$ is an orthonormal basis of $\re^{a+k+c+2}$ consisting of eigenvectors of $A(X)$. 
For $j=6,\ldots, a+k+c+2$, the eigenvector $u_j$ is associated to eigenvalue 0 (as $\operatorname{rk}(A(X))=4$), 
and 
its components on each cell of $\mathcal{P}$ sum up to 0 (as $\hat{P}^Tu_j=0$).
Therefore  the components of $u_j$ corresponding to vertices 0 and 1 are both 0, as the two vertices are singletons of the partition. 
Since $u_1=Pv_1$ also has the two corresponding components 0,  vertex 0 and vertex 1 have the same eigenvalue support, and the support is $\Phi=\{\theta_2,\theta_3,\theta_4,\theta_5\}$. 
Furthermore, as $\theta_2,\theta_3,\theta_4$ and $\theta_5$ are all simple eigenvalues of $A(X)$, we know that vertex 0 and vertex 1 are parallel 
and orthogonal projection onto the eigenspace associated to $\theta_j$ of $A(X)$ is $E_j=u_ju_j^T=\hat{P}v_jv_j^T\hat{P}^T$ for $j=2,3,4,5$. 
Now  observe that  $\frac{(E_r)_{0,0}-(E_r)_{1,1}}{(E_r)_{0,1}}$ is a fixed number 
for $r=2,3,4,5$;  
by Theorem~\ref{thm:commuindal}, the induced algebra  $\cA(\{0,1\})$ is commutative. 
Since $(E_j)_{0,1}>0$ for $j=4,5$ and $(E_j)_{0,1}<0$ for $j=2,3$, 
we have $C^+ =\{\theta_4,\theta_5\}$ and $C^-=\{\theta_2,\theta_3\}$.
\qed
\end{proof}

\smallbreak\noindent
For a rational number $x$, let $|x|_2$ denote the 2-adic norm of $x$.
\begin{lemma}\label{lem:FReigensq}
Let $X=X(a,k,c)$ and let $\theta_3$ and $\theta_5$ be the two eigenvalues of $X$ as given in \eqref{eq:eigenvalues}. 
Then $X$ admits fractional revival on $\{0,1\}$ if and only if  $\theta_3^2$ and $\theta_5^2$ are both integers with the same square-free part $\De$ and  further proper fractional revival occurs if and only if $\left|\frac{\theta_3}{\sqrt{\De}}\right|_2\neq \left|\frac{\theta_5}{\sqrt{\De}}\right|_2$.
\end{lemma}

\begin{proof}
From Lemma~\ref{lem:SFC} we know that 
the induced algebra $\cA(\{0,1\})$ is commutative
 and vertices 0 and 1 are parallel 
with the same eigenvalue support $C^+\cup C^-$, where $C^+=\{\theta_5,\theta_4=-\theta_5\}$ and $C^-=\{\theta_3,\theta_2=-\theta_3\}$. 
Condition $(c)$ in Theorem~\ref{thm:FR} is equivalent to  both $2\theta_3$ and $2\theta_5$ being integer multiples of $\sqrt{\De}$ for some  square-free integer $\De$,
which holds if and only if  
\[
\frac{\theta_3}{\sqrt{\De}}\in \ints, \frac{\theta_5}{\sqrt{\De}}\in \ints. 
\]
In fact, for $j=3,5$,
there exists an integer $m_j$ such that $2\theta_j=m_j \sqrt{\De}$
 if and only if $4\theta_j^2=m_j^2 \De$, 
if and only if $\theta_j^2=m_j^2 \De/4$.
As $\theta_j$ is an algebraic integer and $m_j^2 \De/4$ is a rational number, we conclude that $\theta_j^2=m_j^2 \De/4$ is an integer. 
Since $\De$ is square-free, we know $m_j^2$ is an even integer, 
 hence $m_j$ is an even integer and $\theta_j$ is an integer multiple of $\sqrt{\De}$.   
Now assume that 
\[
\theta_3=\al\sqrt{\De},\;\; \theta_5=\be\sqrt{\De}
\]
for some integers $\al$ and $\be$.  
Since fractional revival occurs if and only if condition $(a),(b)$ and $(c)$ of Theorem~\ref{thm:FR} hold, we have proved the first claim. 
For proper fractional revival to occur, we just need to show that condition $(d)$ in Theorem~\ref{thm:FR} holds if and only if the multiplicities of 2 as a factor of $\al$ and $\be$ are different. 
With the notation of Theorem~\ref{thm:FR}
\begin{align*}
g&=\gcd\left\{\frac{\theta_r-\theta_s}{\sqrt{\De}}\; : \, r,s\in C^+, \text{ or } r,s\in C^-\right\}\\
&=\gcd\left\{\frac{\theta_3-\theta_2}{\sqrt{\De}},\frac{\theta_5-\theta_4}{\sqrt{\De}}\right\}\\
&=2\gcd\left\{\al,\be\right\}
\end{align*}
and $\frac{\rho_3-\om_5}{g}=\frac{\al-\be}{2\gcd\{\al,\be\}}$, which is not an integer if and only if the multiplicities of 2 as a factor in $\al$ and $\be$ are distinct, that is, $|\al|_2\neq |\be|_2$.
In this case, 
the minimum time for proper fractional revival to occur is $\tau=\frac{\pi}{\gcd\{\al,\be\}\sqrt{\De}}$.
\qed
\end{proof}

Recall that a graph $X$ is said to be periodic if there is a time $t$ 
such that $D_a$ is periodic at time $t$ for all $a\in V(X)$.
\begin{corollary}\label{coro:FRimpperi}
If $X(a,k,c)$ admits fractional revival between vertices 0 and 1, 
then $X$ is periodic with minimum period $t=\frac{2\pi}{g\sqrt{\De}}$, 
where $g=\gcd\left\{ \theta_3/ \sqrt{\De}, \theta_5/ \sqrt{\De}\right\}$.
\end{corollary}
\begin{proof}
From Corollary 3.3 of \cite{Periodic}, 
we know that a graph $X$ with adjacency matrix $A=\sum_{r=1}^m\theta_r E_r$ is periodic if and only if 
all its eigenvalues are integer multiples of $\sqrt{\De}$ for a square-free integer $\De$ ($\De=1$ if all the eigenvalues are integers); in this case the minimum period for $X$ is  
$t=(2\pi)/(g\sqrt{\De})$,
where $g$ is the greatest common divisor of 
$\{\theta_r / \sqrt{\De}\,\}_{r=1,\ldots,m}$.\\
The graph $X(a,k,c)$ has exactly 5 distinct eigenvalues as shown in \eqref{eq:eigenvalues}, 
with $\theta_2,\theta_3=-\theta_2,\theta_4,\theta_5=-\theta_4$ all simple, 
and $\theta_1 = 0$. 
Now the result follows from Lemma~\ref{lem:FReigensq}.\qed
\end{proof}

\begin{corollary}\label{coro:frperi}
If proper fractional revival occurs on $\{0,1\}$ in $X=X(a,k,c)$ at time $\tau$, 
then it also occurs at $(2j+1)\tau$ for all positive integer $j$, 
and $X$ is periodic at $2j\tau$ for all positive integer $j$. 
\end{corollary}
\begin{proof}
From Lemma~\ref{lem:FReigensq} and Corollary~\ref{coro:FRimpperi}, 
we know that the minimum time for fractional revival on $\{0,1\}$ to occur in $X(a,k,c)$ 
is $\tau=(\pi)/(g\sqrt{\De})$ and the minimum period is $t=2\tau=(2\pi)/(g\sqrt{\De})$, 
where $g=\gcd\left\{ \theta_3/ \sqrt{\De}, \theta_5/ \sqrt{\De}\right\}
=\gcd\left\{\al,\be\right\}$. 
Both $U(\tau)$ and $U(2\tau)$ are block diagonal, 
with $U(2\tau)_{\{0,1\},\{0,1\}}$ being a scalar matrix, 
therefore $U(2j\tau)$ is block diagonal with the $2\times 2$ block being a scalar matrix,  
 and $U((2j+1)\tau)$ is block diagonal with the $2\times 2$ block not a scalar matrix. 
 \qed
 \end{proof}

Note that there exist parameters $a,k,c$ such that the graph $X(a,k,c)$ is periodic, 
but does not admit proper fractional revival.
\begin{example}
Let $X_1=X(1,4,1)$, then $\theta_3=1$ and $\theta_5=5$. 
Therefore $|\theta_3|_2=|\theta_5|_2=1$. 
So $X$ is periodic but does not admit proper fractional revival, by Lemma~\ref{lem:FReigensq} and Corollary~\ref{coro:FRimpperi}.
Similarly,  $X_2=X(1,16,25)$ is also such a graph, with $\theta_3=3$ and $\theta_5=7$.

\end{example}

\begin{corollary}\label{coro:FRagain}
$X(a,k,c)$ admits proper fractional revival if and only if there exist integers $\al,\be$ with $|\al|_2\neq|\be|_2$ and an square-free integer $\De$, such that 
\begin{eqnarray}
\De(\be^2-\al^2)=\sqrt{4k^2+(a-c)^2},\label{eqn:sumsofsquares}\\
\De(\al^2+\be^2)=2k+a+c.\label{eqn:differofsquares}
\end{eqnarray}
\end{corollary}
\begin{proof}
From Lemma~\ref{lem:FReigensq} we know that $X$ admits proper fractional revival on $\{0,1\}$ if and only if 
\begin{equation}\label{eq:eig3}
\theta_3^2=\frac12(2 k + a + c - \sqrt{4 k^2 + (a-c)^2})=\al^2\De\end{equation}
 and 
\begin{equation}\label{eq:eig5}
\theta_5^2=\frac12(2 k + a + c + \sqrt{4 k^2 + (a-c)^2})=\be^2\De
\end{equation}
for some integers $\al, \be$ and $\De$, with $\De$ square-free, and $|\al|_2\neq|\be|_2$.

The result follows as equations \eqref{eq:eig3} and \eqref{eq:eig5} are equivalent to \eqref{eqn:sumsofsquares} and \eqref{eqn:differofsquares}. \qed
\end{proof}

Now we show that there are triples $(a,k,c)$ such that $X(a,k,c)$ admits proper fractional revival.

\begin{theorem}
There exist graphs which admit proper fractional revival  between non-cospectral vertices, and there are infinitely many such graphs. 
\end{theorem}
\begin{proof}
We prove the result by showing that there are infinitely many positive integer triples $(a,k,c)$ such that proper fractional revival on $\{0,1\}$ occurs in $X(a,k,c)$ with vertices 0 and 1 non-cospectral. \\ 
Note that vertices 0 and 1 are non-cospectral in $X(a,k,c)$ if and only if $a\ne c$. 
Without loss of generality, assume $a<c$. 
Now we make use of Corollary~\ref{coro:FRagain}.
Let $p$ be a prime such that $p\equiv1\pmod{4}$, then by Fermat's Theorem on sums of two squares, 
$p=f^2+g^2$ 
for some positive integers $f>g$. 
Now we pick integers $\De$ (square-free) and $\al < \be$,  
such that $|\al|_2\neq|\be|_2$ and $p|\De(\be^2-\al^2)$. Let $d$ be the integer satisfying $\De(\be^2-\al^2)=pd=(f^2+g^2)d$. 
Then the triple $(a,k,c)$ with 
\begin{align*}
&k=fgd\\
&c=\frac{\De(\be^2+\al^2)+(f^2-g^2)d}{2}-fgd=\De \al^2+fd(f-g)\\
&a=\frac{\De(\be^2+\al^2)-(f^2-g^2)d}{2}-fgd=\De \al^2-gd(f-g)
\end{align*}
is a set of integer solutions of equations \eqref{eq:eig3} and  \eqref{eq:eig5}. 
If further 
\[
\frac{\be}{\al}<\sqrt{2}+1
\] 
 then $a,k,c$ are all positive integers.
Let $\tau=\frac{\pi}{\sqrt{\De}\gcd\{\al,\be\}}$, then there is proper fractional revival on $\{0,1\}$ at time $\tau$. 
There are infinitely many such graphs, as we have infinitely many choices for $p,\De,\al,\be$.

In particular, let $p=5$, then $f=2$ and $g=1$. 
If we let 
\begin{itemize}
\item
let $\De=5$, then we can pick any positive integer $\al$ and $\be=2\al$. In this case, $a=2\al^2, k=6\al^2$ and $c=11\al^2$.
\item
let $\be=3$ and $\al=2$, then we can pick any square free integer $\De$. In this case, $a=3\De, k=2\De$ and $c=6\De$.
\item
let $\be=6$ and $\al= 4$, then we can pick any square free integer $\De$. In this case, $a=12\De, k=8\De$ and $c=24\De$.
\item
let $\be=8$ and $\al=7$, then we can pick any square free integer $\De$.  In this case, $a=46\De, k=6\De$ and $c=55\De$. \qed
\end{itemize} 
\end{proof}

\begin{example}
Graph $X(3,2,6)$ in Figure~\ref{fig:FRnoncos} admits proper fractional revival between  (non-cospectral) vertices 0 and 1 at time $\pi$. Other examples include:\\
$X(6,3,14)$ at $ t = \pi / \sqrt{2}$,  
$X(6,4,12)$ at $t = \pi/ \sqrt{2}$, 
$X(9,6,18)$ at  $t = \pi/\sqrt{3}$, 
$X(2,6,11)$ at $ t = \pi/\sqrt{5}$,  
$X(2,6,28)$ at  $t = \pi/2$.
\begin{figure}[h]
\begin{center}
\definecolor{ududff}{rgb}{0.30196078431372547,0.30196078431372547,1.}
\begin{tikzpicture}[line cap=round,line join=round,>=triangle 45,x=1.0cm,y=1.0cm,scale=1.0]
\clip(0,1.6) rectangle (10.6,4.3);
\draw [line width=1.pt] (1.,4.)-- (3.,3.);
\draw [line width=1.pt] (1.,3.)-- (3.,3.);
\draw [line width=1.pt] (1.,2.)-- (3.,3.);
\draw [line width=1.pt] (3.,3.)-- (5.,4.);
\draw [line width=1.pt] (3.,3.)-- (5.,2.);
\draw [line width=1.pt] (5.,4.)-- (7.,3.);
\draw [line width=1.pt] (5.,2.)-- (7.,3.);
\draw [line width=1.pt] (7.,3.)-- (9.02,4.2);
\draw [line width=1.pt] (7.,3.)-- (9.02,3.74);
\draw [line width=1.pt] (7.,3.)-- (9.02,3.3);
\draw [line width=1.pt] (7.,3.)-- (9.04,2.74);
\draw [line width=1.pt] (7.,3.)-- (9.02,2.14);
\draw [line width=1.pt] (7.,3.)-- (9.02,1.72);

\draw (2.8,2.8) node[anchor=north west] {$0$};
\draw (6.8,2.8) node[anchor=north west] {$1$};
\begin{scriptsize}
\draw [fill=ududff] (1.,3.) circle (2.5pt);
\draw [fill=ududff] (1.,2.) circle (2.5pt);
\draw [fill=ududff] (1.,4.) circle (2.5pt);
\draw [fill=ududff] (3.,3.) circle (2.5pt);
\draw [fill=ududff] (5.,4.) circle (2.5pt);
\draw [fill=ududff] (5.,2.) circle (2.5pt);
\draw [fill=ududff] (7.,3.) circle (2.5pt);
\draw [fill=ududff] (9.02,3.3) circle (2.5pt);
\draw [fill=ududff] (9.02,3.74) circle (2.5pt);
\draw [fill=ududff] (9.02,4.2) circle (2.5pt);
\draw [fill=ududff] (9.04,2.74) circle (2.5pt);
\draw [fill=ududff] (9.02,2.14) circle (2.5pt);
\draw [fill=ududff] (9.02,1.72) circle (2.5pt);
\end{scriptsize}
\end{tikzpicture}
\caption{A graph with proper fractional revival between non-cospectral vertices 0 and 1}
\label{fig:FRnoncos}
\end{center}
\end{figure}
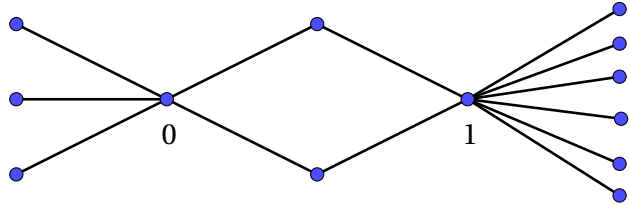
\end{example}

\begin{remark}
{\rm
When $k=1$, \eqref{eqn:sumsofsquares} implies that 4 is the difference of two nonzero integer squares (if $a\ne c$) or that $2=\De(\be^2-\al^2)$ with $\De,\al,\be\in\ints$ (if $a=c$). 
But there are no such integer pairs. 
Therefore, tree graphs in this family do not admit proper fractional revival. 
}
\end{remark}
Nevertheless, there is an infinite family of trees that admits proper fractional revival between cospectral vertices: Let $X$ be the graph obtained by connecting the two centers of two copies of the star $K_{1,a}$. Label the two centers vertices 0 and 1, respectively. Then $X$ admits proper fractional revival on $\{0,1\}$ at time $t=\frac{2\pi}{\sqrt{4a+1}}$.

\subsection{Commutative algebra}

In Lemma ~\ref{lem:SFC}, we show that for the graph $X=X(a,k,c)$ the induced algebra 
on $\{0,1\}$ is commutative by using the spectral decomposition of $A(X)$ and verifying that $\frac{(E_r)_{0,0}-(E_r)_{1,1}}{(E_r)_{0,1}}$ is independent of $r$ (given that $\theta_r$ lies
in the eigenvalue support of the vertex $0$). As a computational technique, this has serious 
difficulties because the entries of $E_r$ will be floating point numbers. We show how
to carry out this computation in exact arithmetic, using the polynomial condition in Theorem~\ref{thm:commuindal}(d).

The characteristic polynomial of $X(a,k,c)$ is:
\[
	t^{a+k+c-2} (t^4 - (a+2k+c)t^2 +ak +ck +ac).
\]
Recall that the vertex of degree $a+k$ is denoted by $0$, and the vertex of degree $c+k$
is denoted by $1$. 
The characteristic polynomials of $X\diff0$ and $X\diff1$ are, respectively,
\[
	t^{a+k+c-1}(t^2-c-k),\quad t^{a+k+c-1}(t^2-a-k).
\]
Finally
\[
	\phi(X\diff\{0,1\},t) = t^{a+k+c}, \quad\psi_{0,1}(X,t) = kt^{a+k+c-1}.
\]

We see that $\phi(X\diff 0,t)-\phi(X\diff 1,t)=\frac{a-c}{k}  \psi_{0,1}(X,t)$, 
therefore by Theorem~\ref{thm:commuindal},
the induced algebra of $X$ on $\{0,1\}$ is commutative. 

From 
\[
\left( (tI-A)^{-1}\right)_{a,b}=\frac{\left(\operatorname{adj}(tI-A)\right)_{b,a}}{\phi(X,t)} 
\text{ and }
\left( (tI-A)^{-1}\right)_{a,b}=\sum_r \frac{(E_r)_{a,b}}{t-\theta_r},
\]
 we get that 
for each $r$,
\[
	(E_r)_{\{0,1\},\{0,1\}}=
	\lim_{t\to\theta_r}\frac{t-\theta_r}{\phi(X,t)}
		\begin{bmatrix}
		 \phi(X\diff0,t)& \psi_{0,1}(X,t)\\ \psi_{0,1}(X,t)& \phi(X\diff1,t)
		 \end{bmatrix}
\]
Note that $(E_r)_{1,1}$ is the coefficient of $1/(t-\theta_r)$ in the partial
fraction expansion of 
$
	\frac{\phi(X\diff 1,t)}{\phi(X,t)},
$
and the other entries have similar expressions.

By way of example, for the triple $(3,2,6)$,
\begin{align*}
	\frac{\phi(X\diff0,t)}{\phi(X,t)} 
		&= \frac{1}{10}\left(\frac1{t - 3}+ \frac4{t - 2}+ \frac{4}{t + 2}
			+\frac1{t + 3}\right)\\
	\frac{\psi_{0,1}(t)}{\phi(X,t)}
		&= \frac2{10}\left(\frac1{t - 3}- \frac1{t-2}-\frac1{t+2}
			+\frac1{t+3}\right)\\
	\frac{\phi(X\diff1,t)}{\phi(X,t)} 
		&= \frac{1}{10}\left(\frac4{t - 3}+ \frac1{t - 2}+ \frac{1}{t + 2}
			+\frac4{t + 3}\right).
\end{align*}
The eigenvalues of $X$, 
as denoted in \eqref{eq:eigenvalues}, 
are 
$\theta_1=0, \, \theta_2=-2,\,  \theta_3=2,\, \theta_4=-3,\, \theta_5=3$, 
and 
\[
	(E_2)_{\{0,1\},\{0,1\}}=(E_3)_{\{0,1\},\{0,1\}}=\frac1{10}\begin{bmatrix}4&-2\\-2&1\end{bmatrix},\qquad
	(E_4)_{\{0,1\},\{0,1\}}=(E_5)_{\{0,1\},\{0,1\}}=\frac1{10}\begin{bmatrix}1&2\\2&4\end{bmatrix}.
\]
For any positive integer $m$, we get exactly the same matrices for the triples $(3m,2m,6m)$, and so all graphs
in this family admit proper fractional revival.

\section{No balanced fractional revival on $\{0,1\}$ in $X(a,k,c)$}
Again, let $a\leq c$, here we consider whether balanced fractional revival on $\{0,1\}$ could occur in $X(a,k,c)$, that is, whether there is a time $\tau$ such that $|U(\tau)_{0,0}|=|U(\tau)_{0,1}|=1/\sqrt{2}$. 
Let $X$ be a simple connected graph. 
Assume proper fractional revival occurs  on $\{1,2\}$ at time $\tau$ in $X$. 
Then $U(\tau)_{\{1,2\},\{1,2\}}$ is unitary. Take  $\theta_1\in C^+, \theta_2\in C^-$, 
then by equation \eqref{eq:profrspec} 
\[
U(\tau)_{\{1,2\},\{1,2\}}
=  e^{i\tau\theta_1}\begin{bmatrix}
p^2+e^{i\tau\theta}q^2   &   pq(1-e^{i\tau\theta}) \\
pq(1-e^{i\tau\theta})       &   q^2+e^{i\tau\theta}p^2
\end{bmatrix},
\]
for some positive numbers $p\geq q$ with $p^2+q^2=1$, 
where $\theta=\theta_2-\theta_1$. 

Now, balanced fractional revival occurs at time $\tau$ if and only if 
$\frac12
=|pq(1-e^{i\tau\theta})|^2
=2p^2q^2(1-\cos(\tau\theta))$,
which implies that 
\begin{align}\label{eq:ratsquarecos}
\cos(\tau\theta)& =1-\frac{1}{4p^2q^2} \nonumber \\
& =\frac{4p^2q^2-(p^2+q^2)^2}{4p^2q^2}  \quad (\text{ as } p^2+q^2=1) \nonumber \\
&= -\frac{(p^2-q^2)^2}{4p^2q^2}  \nonumber \\
&= -\frac14(\frac pq-\frac qp)^2 \in\mathbb{Q}, 
\end{align}
where the rationality comes from the fact that 
$(E_r)_{1,1}-(E_r)_{2,2}=(\frac pq-\frac qp)(E_r)_{1,2}$ for each $r$,
and therefore $\frac pq-\frac qp$ is the rational scalar 
 $\ga$ in Theorem~\ref{thm:commuindal}.

Now assume that $X$ is also periodic at vertices $1,2$, say at time $\tau_1$. Then $e^{i\tau_1\theta_1}=e^{i\tau_1\theta_2}$, that is $e^{i\tau_1\theta}=1$,
or equivalently, $\tau_1\theta=2k\pi$ for some integer $k$.
Assume that the minimum time when proper fractional revival occurs on $\{1,2\}$ is $\tau_0$, 
then $\tau = r \tau_0$ and $\tau_1 = r_1 \tau_0$ 
for some $r,r_1\in\ints$. 
Therefore $\frac{\tau}{\tau_1}\in\mathbb{Q}$, and $\tau\theta$ is a rational multiple of $\pi$, with $\cos(\tau\theta)\in\mathbb{Q}$, which implies that $\cos(\tau\theta)\in\{0,-\frac12,-1\}$ as $\cos(\tau\theta)\leq 0$. 
But equation \eqref{eq:ratsquarecos} in fact tells us that $\sqrt{-\cos(\tau\theta)}\in\mathbb{Q}$. Therefore $\cos(\tau\theta)\in\{0,-1\}$. 

If $\cos(\tau\theta)=0$, 
then $\ga=\frac pq-\frac qp=0$ and vertices $1$ and $2$ are cospectral. It can be shown that $X$ admits perfect state transfer between $1$ and $2$ at time $2\tau$ in this case \cite{fr}.

If $\cos(\tau\theta)=-1$, 
then $e^{i\tau\theta}=-1$, $e^{i2\tau \theta}=1$ and 
$X$ is periodic at vertices $1,2$ at time $2\tau$. In this case, 
$\ga=\frac pq-\frac qp=2$, 
that is, $(A^k_{1,1}-A^k_{2,2})=2A^k_{1,2}$ for any positive integer $k$. In fact, $\frac pq=1+\sqrt{2}$ in this case. 
The achieved results are summarized in the following lemma. 
\begin{lemma}\label{lem:balFR}
Assume that graph $X$ admits balanced fractional revival  on $\{a,b\}$ at time $\tau$ 
and  non-proper fractional revival on $\{a,b\}$ at time $t$. Then either vertices $a$ and $b$ are cospectral and $X$ admits perfect state transfer between them at time $2\tau$, or $a$ and $b$ are non-cospectral and $X$ is periodic at them at time $2\tau$. In the second case, $(A^k)_{a,a}-A^k_{b,b}=2A^k_{a,b}$ for all nonnegative integer $k$, 
that is, the scalar $\ga$ in Theorem~\ref{thm:commuindal} satisfies $\ga=2$.
\end{lemma}
\begin{corollary}\label{coro:nobalFRakc}
$X(a,k,c)$ does not admit balanced fractional revival on $\{0,1\}$ for any positive integers $a,k,c$.
\end{corollary}
\begin{proof}
From Lemma~\ref{lem:balFR} we know that if balanced fractional revival occurs between vertices 0 and 1, then $\frac pq=1+\sqrt{2} \notin \mathbb{Q}$. From equation~\eqref{eq:eigvec} we know $\frac pq =\frac{-c+a-\sqrt{4k^2+(a-c)^2}}{2k}$, which is a rational number by Corollary~\ref{coro:FRagain}, a contradiction.
\end{proof}

\section{Fractional revival is polygamous}
Let $X$ be a graph. 
If there is perfect state transfer from $a$ to $b$ and from $a$ to $c$, 
then $b=c$. 
For fractional revival, 
weighted graphs where there is proper fractional revival on $\{a,b\}$ and on $\{a,c\}$ with $a,b,c$ distinct vertices of $X$ were constructed \cite{FFR}, 
but no unweighted graph examples were known. 
Here we construct unweighted graphs where there is proper fractional revival on $\{a,b\}$ and on $\{a,c\}$ for $b\ne c$. 

Assume that $X$ is a graph on $m$ vertices with adjacency matrix $A(X)$, 
and $Y$ is a graph on $n$ vertices with adjacency matrix $A(Y)$. 
Their Cartesian product $X\square Y$ is the graph with adjacency matrix 
$A(X\square Y)=I_m\otimes A(Y) + A(X)\otimes I_n$, 
and the transition matrix of continuous time quantum walk on $X\square Y$ at time $\tau$ is 
\begin{equation}\label{eq:cartU}
U_{X\square Y}(\tau)=U_X(\tau)\otimes U_Y(\tau).
\end{equation}
Therefore if there is a time $t$ such that $X$ is periodic at vertex $x_1$ and graph $Y$ admits proper fractional revival on $\{y_1,y_2\}$, 
then $X\square Y$ admits proper fractional revival on $\{(x_1,y_1),(x_1,y_2)\}$ at time $t$. 
Now we take $X$ to be  $K_2$,
which is periodic at integer multiples of $\pi$, 
admits perfect state transfer at odd integer multiples of $\pi/2$, 
and admits proper fractional revival at all other time. 
By \eqref{eq:cartU} 
we know that if for some positive integer $\ell$ we can find a graph $Y$ that admits proper fractional revival on $\{a,b\}$ at time $\tau=\frac{\pi}{2\ell+1}$ and is periodic at vertices $a,b$ at time $2\tau$, 
then $K_2\square Y$ admits fractional revival between vertices $(0,a)$ and $(1,a)$ at time $2\tau$, 
and admits proper fractional revival on $(0,a)$ and $(0,b)$ at time $t = \pi=(2\ell+1)\tau$. 
 
Corollary \ref{coro:frperi} tells us that 
if proper fractional revival occurs on $\{0,1\}$ at time $\tau$ in $X(a,k,c)$, then it also occurs at time $(2j+1)\tau$, and $X(a,k,c)$ is periodic at vertices 0 and 1  at time $2j\tau$ for any positive integer $j$.
The graph $X_1=X(16,36,37)$ admits fractional revival between vertex 0 and 1 at $\frac{\pi}{5}$, 
the graph $X_2=X(10,30,55)$ admits fractional revival at $\frac{\pi}{5}$, 
and the graph $X_3=X(27,18,54)$ admits fractional revival at time $\frac{\pi}{3}$. 
In fact there are infinitely many such graphs.  

\begin{theorem}
There are infinitely many unweighted graphs where the fractional revival pairs overlap (fractional revival is polygamous).
\end{theorem}
\begin{proof}
We show that there are  infinitely many graphs in the family $X=X(a,k,c)$ such that $X$ admits fractional revival at time $\frac{\pi}{2\ell+1}$ for some integer $\ell$. 
Let $p$ be a prime such that $p\equiv 1\pmod{4}$. Then $p = f^2+g^2$ for some $f>g$ by Fermat's theorem. Let $r$ be any positive integer.  Then the graph $X(a,k,c)$ with  
\[
a=p^2r^2-gp(2r+1)(f-g), k=fgp(2r+1), c=p^2r^2+fp(2r+1)(f-g)
\]
 admits proper fractional revival at time $\tau=\frac{\pi}{p}$. 
 By the argument before the theorem, 
 we know that $K_2\square X$ admits proper fractional revival with overlapping pairs. 
 \qed
\end{proof}

\section{Further questions}

We showed in Corollary~\ref{coro:nobalFRakc} that 
balanced fractional revival does not occur in $X(a,k,c)$, 
the only class of graphs known where proper fractional revival occur between non-cospectral vertices. Apart from looking for other families,  
we wonder: 
\begin{question}
Is there a tree $T$ where there is proper fractional revival between two non-cospectral vertices $a$ and $b$?
\end{question}

\begin{question}
Can balanced fractional revival occur between non-cospectral vertices?
\end{question}
Assume proper fractional revival occurs on $\{a,b\}$. 
Then there exists a square-free integer $\De$
such that $\frac{\theta_r-\theta_s}{\sqrt{\De}}\in\ints$ for any $\theta_r,\theta_s\in C^+$ or any  $\theta_r,\theta_s\in C^-$.  
If further $a$ and $b$ are cospectral, 
then the two sets $C^+$ and $C^-$ are both closed under taking algebraic conjugates, 
and therefore there exists a square-free integer $\De$ and integers $a^+,a^-,b_r$
such that  $\theta_r=\frac{a^+ +b_r\sqrt{\De}}{2}$ for any $\theta_r\in C^+$,
and $\theta_s=\frac{a^- +b_r\sqrt{\De}}{2}$ for any $\theta_s\in C^-$ \cite{fr}. 
When proper fractional revival occurs on $X(a,k,c)$, 
the eigenvalues in the support are quadratic integers as in the cospectral case. 
Are there unweighted graphs with proper fractional revival such that 
$C^+$ and $C^-$ are not closed under taking algebraic conjugates (then it occurs between non-cospectral vertices), or 
\begin{question}
Is there a graph $X$ with proper fractional revival on $\{a,b\}$ such that
eigenvalues in the support of $a$ are not integers or quadratic integers?
Vertices $a$ and $b$ must be noncospectral in $X$ if it happens.
\end{question}


\begin{bibdiv}
\begin{biblist}

\bib{circ}{article}{
   author={Ba\v{s}i\'{c}, Milan},
   title={Characterization of quantum circulant networks having perfect
   state transfer},
   journal={Quantum Inf. Process.},
   volume={12},
   date={2013},
   number={1},
   pages={345--364},
      }

\bib{gwithfr}{article}{
   author={Bernard, Pierre-Antoine},
   author={Chan, Ada},
   author={Loranger, \'{E}rika},
   author={Tamon, Christino},
   author={Vinet, Luc},
   title={A graph with fractional revival},
   journal={Phys. Lett. A},
   volume={382},
   date={2018},
   number={5},
   pages={259--264},
}

\bib{grouptransfer}{article}{
      title={Continuous time quantum walks on graphs: group state transfer}, 
      author={Brown, Luke C.}, 
      author={Martin, William J.}, 
      author={Wright, Duncan },
      year={2021},
      eprint={2103.08837},
      archivePrefix={arXiv},
}


\bib{FFR}{article}{
   author={Chan, Ada},
   author={Coutinho, Gabriel},
   author={Drazen, Whitney},
   author={Eisenberg, Or},
   author={Godsil, Chris},
   author={Lippner, Gabor},
   author={Kempton, Mark},
   author={Tamon, Christino},
   author={Zhan, Hanmeng},
   title={Fundamentals of fractional revival in graphs},
    year={2020},
    eprint={	arXiv:2004.01129 [math.CO]}
}

\bib{fr}{article}{
   author={Chan, Ada},
   author={Coutinho, Gabriel},
   author={Tamon, Christino},
   author={Vinet, Luc},
   author={Zhan, Hanmeng},
   title={Quantum fractional revival on graphs},
   journal={Discrete Appl. Math.},
   volume={269},
   date={2019},
   pages={86--98},
}



\bib{cubelike}{article}{
   author={Cheung, Wang-Chi},
   author={Godsil, Chris},
   title={Perfect state transfer in cubelike graphs},
   journal={Linear Algebra Appl.},
   volume={435},
   date={2011},
   number={10},
   pages={2468--2474},
}

\bib{PST05}{article}{
 author = {Christandl, Matthias},
 author = {Datta, Nilanjana},
 author = {Dorlas, Tony C.},
 author = {Ekert, Artur},
 author = {Kay, Alastair},
 author ={Landahl, Andrew J.},
 title = {Perfect transfer of arbitrary states in quantum spin networks},
  journal = {Phys. Rev. A},
  volume = {71},
  issue = {3},
  pages = {032312},
  numpages = {11},
  year = {2005},
  month = {Mar},
   }


\bib{thebook}{article}{
   author={Coutinho, Gabriel},
   author={Godsil, Chris},
   title={Graph spectra and continuous quantum walks},
   journal={Lecture notes},
   date={2021},
}



\bib{drg}{article}{
   author={Coutinho, Gabriel},
   author={Godsil, Chris},
   author={Guo, Krystal},
   author={Vanhove, Fr\'{e}d\'{e}ric},
   title={Perfect state transfer on distance-regular graphs and association
   schemes},
   journal={Linear Algebra Appl.},
   volume={478},
   date={2015},
   pages={108--130},
}

\bib{frchain}{article}{
   author={Genest, Vincent X.},
   author={Vinet, Luc},
   author={Zhedanov, Alexei},
   title={Exact fractional revival in spin chains},
   journal={Modern Phys. Lett. B},
   volume={30},
   date={2016},
   number={26},
   pages={1650315, 7},
}

\bib{frchain2}{article}{
   author={Genest, Vincent X. },
   author={Vinet, Luc},
   author={Zhedanov, Alexei},
   title={Quantum spin chains with fractional revival},
   journal={Ann. Physics},
   volume={371},
   date={2016},
   pages={348--367},
}



\bib{algcomb}{book}{
   author={Godsil, C. D.},
   title={Algebraic {C}ombinatorics},
   series={Chapman and Hall Mathematics Series},
   publisher={Chapman \& Hall, New York},
   date={1993},
   pages={xvi+362},
   isbn={0-412-04131-6},
}

\bib{Periodic}{article}{
   author={Godsil, Chris},
   title={Periodic graphs},
   journal={Electron. J. Combin.},
   volume={18},
   date={2011},
   number={1},
   pages={Paper 23, 15},
}

\bib{godsil2017real}{article}{
      title={Real state transfer}, 
      author={Godsil, Chris },
      year={2017},
      eprint={1710.04042},
      archivePrefix={arXiv},
      }

\bib{transfer}{article}{
   author={Godsil, Chris},
   title={State transfer on graphs},
   journal={Discrete Math.},
   volume={312},
   date={2012},
   number={1},
   pages={129--147},
}

\bib{strongcos}{article}{
      title={Strongly cospectral vertices}, 
      author={Godsil, Chris }
      author={Smith, Jamie},
      year={2017},
      eprint={1709.07975},
      archivePrefix={arXiv},
}

\bib{pstmono}{article}{
  author = {Kay, Alastair},
  title = {Basics of perfect communication through quantum networks},
  journal = {Phys. Rev. A},
  volume = {84},
  issue = {2},
  pages = {022337},
  numpages = {8},
  year = {2011},
  month = {Aug},
}


\end{biblist}
\end{bibdiv}
\end{document}